\documentclass[12pt,draft]{amsproc}

\usepackage[english]{babel}
\usepackage{latexsym}
\usepackage{amssymb}
\usepackage{amscd}
\usepackage[latin1]{inputenc}
\usepackage{mathrsfs}
\usepackage{color}

\newtheorem{teor}{Theorem}[section]
\newtheorem{prop}[teor]{Proposition}
\newtheorem{cor}[teor]{Corollary}
\newtheorem{lema}[teor]{Lemma}

\theoremstyle{definition}
\newtheorem{defin}[teor]{Definition}

\newtheorem{example}[teor]{Example}

\theoremstyle{remark}

\newtheorem{remar}[teor]{Remark}

\newcommand{\fin}{\hspace*{\fill} $\Box$}

\newcommand{\N}{\mathbb{N}}

\newcommand{\To}{\longrightarrow}

\newcommand{\U}{\mathscr{U}}

\def\Ext{\operatorname{Ext}}

\def\PO{\operatorname{PO}}
\def\PB{\operatorname{PB}}

\def\fin{\operatorname{fin}}

\newcommand{\aproof}{\begin{proof}}
\newcommand{\zproof}{\end{proof}}

\def\KiU{(K_i)^{\mathscr U}}
\def\XiU{[X_i]_{\mathscr U}}
\def\Xii{(X_i)_{i\in I}}
\def\Kii{(K_i)_{i\in I}}

\def\dens{\operatorname{dens}}

\def\e{\varepsilon}

\def\N{\mathbb{N}}
\def\I{\mathbb{I}}

\def\vp{\varphi}

\def\PB{\operatorname{PB}}
\def\ran{\operatorname{ran}}

\begin{document}

\title{$\aleph$-injective Banach spaces and $\aleph$-projective compacta}

\author[Avil\'es et al.]
{Antonio Avil\'es, F\'elix Cabello S\'anchez,\\ Jes\'us M.
F. Castillo, Manuel Gonz\'alez and Yolanda Moreno}

\address{Departamento de Matem\'aticas, Universidad de Murcia, 30100 Espinardo,
Murcia, Spain} \email{avileslo@um.es}

\address{Departamento de Matem\'aticas, Universidad de Extremadura, Avenida de Elvas s/n,
06011 Badajoz, Spain}      \email{fcabello@unex.es}

\address{Departamento de Matem\'aticas, Universidad de Extremadura, Avenida de Elvas s/n,
06011 Badajoz, Spain}       \email{castillo@unex.es}

\address{Departamento de Matem\'aticas, Universidad de Cantabria, Avenida los Castros s/n,
39071 Santander, Spain}     \email{manuel.gonzalez@unican.es}

\address{Escuela Polit\'ecnica, Universidad de Extremadura, Avenida de la Universidad s/n,
10071 C\'aceres, Spain}        \email{ymoreno@unex.es}

\begin{abstract}
A Banach space $E$ is said to be injective if for every Banach space $X$ and every subspace $Y$ of $X$
every operator $t:Y\to E$ has an extension $T:X\to E$.
We say that $E$ is $\aleph$-injective (respectively, universally $\aleph$-injective) if the preceding
condition holds for Banach spaces $X$ (respectively $Y$) with density less than a given uncountable
cardinal $\aleph$.
We perform a study of $\aleph$-injective and universally $\aleph$-injective Banach spaces which extends
the basic case where $\aleph=\aleph_1$ is the first uncountable cardinal.
When dealing with the corresponding  ``isometric'' properties we arrive to our main examples:
ultraproducts and spaces of type $C(K)$.
We prove that ultraproducts built on countably incomplete $\aleph$-good ultrafilters are
$(1,\aleph)$-injective as long as they are Lindenstrauss spaces. We characterize $(1,\aleph)$-injective
$C(K)$ spaces as those in which the compact $K$ is an $F_\aleph$-space (disjoint open subsets which
are the union of less than $\aleph$ many closed sets have disjoint closures) and we uncover some
projectiveness properties of $F_\aleph$-spaces.
\end{abstract}

\thanks{The research of the last four authors has been supported in part by project MTM2010-20190.
That of the second, third and fifth authors by the program Junta de Extremadura GR10113 IV Plan
Regional I+D+i, Ayudas a Grupos de Investigaci\'on.}

\thanks{AMS (2010) Class. Number. 46B03, 54B30, 46B08, 54C15, 46B26 }



\maketitle

\section*{Introduction}
A Banach space $E$ is said to be injective if for every Banach space $X$ and every subspace $Y$
of $X$, each operator $t \colon Y \to E$ admits an extension $T\colon X \to E$.
In this paper we consider two weak forms of injectivity that arise by limiting the size of either
the subspace $Y$ or the containing space $X$ in the preceding definition.
Let us label them right now. Recall that the density character $\dens(X)$  of a topological
space $X$ is the smallest cardinal a dense subset of $X$ can have.

\begin{defin}\label{def}
Let $E$ be a Banach space, $\aleph$ an uncountable cardinal and $\lambda\geq 1$ a real number.
We say that $E$ is \emph{$\aleph$-injective} if for every Banach space $X$ with
$\dens(X)<\aleph$ and each subspace $Y\subset X$, every operator $t:Y\to E$ can be extended to an
operator $T:X\to E$.

We say that $E$ is \emph{$(\lambda,\aleph)$-injective} if we can always find some extension $T$
such that $\|T\|\leq\lambda\|t\|$.

Replacing the condition $\dens(X)<\aleph$ by $\dens(Y)< \aleph$ one obtains the definitions of
\emph{universally $\aleph$-injective} and \emph{universally $(\lambda,\aleph)$-injective} spaces.
\end{defin}

The choice $\aleph=\aleph_0$ (the first infinite cardinal) would not be too interesting to us:
all Banach spaces would be ``universally $\aleph_0$-injective''.
Asking for an uniform bound on the norm of the extension, that is, considering ``(universally)
$(\lambda,\aleph_0)$-injective'' spaces one arrives to the class of $\mathscr L_\infty$-spaces
which has been widely studied in Banach space theory in connection with the extension of
compact operators; see \cite[Section~4]{zbook}.
Moving to the first uncountable cardinal $\aleph_1$ one obtains the classes of separably injective
and universally separably injective Banach spaces, which also attracted attention.
(Admittedly, that the resulting name for separably injective spaces turns out to be
``$\aleph_1$-injective'' is perhaps surprising.
Nevertheless, we have followed the uses of set theory were properties labeled by a cardinal
$\aleph$ always indicate that something happens for sets whose cardinality is strictly less
than $\aleph$.)
It is worth noticing that Zippin proved in the late seventies that every infinite dimensional
separable and separably injective Banach space has to be isomorphic to $c_0$, the space of all
null sequences with the sup norm, and so even in the case $\aleph=\aleph_1$ one is mainly concerned
with nonseparable Banach spaces.
We refer the reader to \cite{zbook, accgm1} for an account and further references.
\medskip

Let us describe the plan of the paper and highlight its main results.
Section~\ref{pre} is preliminary; it contains some definitions together with the minimal background
on exact sequences of Banach spaces one needs to read the paper.
In Section~\ref{aleph-inj} we extend a variety of results in \cite{accgm1} about (universal)
separably injective Banach spaces to higher cardinals.
However, we found no reasonable generalization for a considerable portion of the results proved
in \cite{accgm1} for $\aleph_1$ and so the resulting picture is rather incomplete.
In contrast to Section~\ref{aleph-inj}, which deals mainly with ``isomorphic''  properties,
the ensuing Section~\ref{1aleph-inj} is of ``isometric nature'' and studies some special
properties of $(1,\aleph)$-injective spaces and their universal relatives.
We use ideas of Lindenstrauss to give a characterization of $(1,\aleph)$-injective spaces by means
of intersection properties of balls (Lemma~\ref{crit}) and we prove that under the generalized
continuum hypothesis $(1,\aleph)$-injective spaces are in fact universally $(1,\aleph)$-injective.
The characterization just mentioned opens the door to the main examples worked in
Sections~\ref{C(K)-spaces} and~\ref{ultra}: spaces of continuous functions on compacta and ultraproducts.
Theorem \ref{people} unifies and extends several characterizations of $(1, \aleph)$-injective
$C(K)$-spaces in terms of properties of the compact space $K$.
The space $\ell_\infty/c_0 =C(\N^*)$ is shown to be $(1,\aleph_1)$-injective but not
$(1, \aleph_2)$-injective.
Theorem~\ref{lind2} establishes that ultraproducts via $\aleph$-good ultrafilters become
$(1,\aleph)$-injective whenever they are Lindenstrauss spaces.
As a corollary we solve a question of Bankston by showing that ultracoproducts of arbitrary compact
spaces over $\aleph$-good ultrafilters are $F_\aleph$-spaces.
The characterization of $(1,\aleph)$-injective $C(K)$ spaces as those in which the compact $K$ is
an $F_\aleph$-space will lead us to study projectiveness properties of these compacta which are
interesting in its own right.
As it is well-known, a Banach space is 1-injective if and only if it is isometrically isomorphic
to $C(K)$ for some extremely disconnected compact space $K$.
On the other hand, such compacta are precisely the projective elements in the category of compacta
and continuous maps, a classical result by Gleason.
This means that if $\sigma: L\to M$ is a continuous surjection then any continuous map $\vp:K\to M$
lifts to $L$ in the sense that there is $\tilde \vp:K\to L$ such that $\vp=\sigma\circ\tilde\vp$.
Motivated by these correspondences, in Section~\ref{projectiveness} we explore the projectiveness
properties of compact $F_\aleph$-spaces since, as mentioned before, a compact space $K$ is an
$F_\aleph$-space precisely when the Banach space $C(K)$ is $(1,\aleph)$-injective.
According to a result of Neville and Lloyd, totally disconnected $F_\aleph$-spaces can be
characterized as those compact spaces which are projective with respect to surjections
$\sigma:L\to M$ between compacta of weight less than $\aleph$.
Theorem~\ref{metricallyprojective} states that this is also equivalent to projectiveness with respect
to compacta that are hereditarily of Lindel\"{o}f number below $\aleph$.
At the end of Section~\ref{projectiveness} we present a characterization of $F_\aleph$-spaces without
any connectedness hypothesis, namely, that a compact space is an $F_\aleph$-space if and only if
it is ``projective'' with respect to all affine surjections between compact convex sets of weight
less that $\aleph$.

We close with a few open problems that arise naturally from the content of the paper and we were
unable to resolve.

\section{Preliminaries}\label{pre}

\subsection{Notations, conventions}
All Banach spaces will be assumed to be real. All the results in this paper can be translated to
the complex case, sometimes with some extra effort, but we have preferred not to do that.

Our notation is fairly standard, as in \cite{lindtzaf}, except perhaps in that given a cardinal
number $\aleph$ we denote by $\ell_\infty(\aleph)$ the space of all bounded functions defined on
an unspecified set $\Gamma$ with $|\Gamma|=\aleph$, endowed with the sup norm and $c_0(\aleph)$
the closed subspace spanned by the characteristic functions of the singletons of $\Gamma$.
By $\ell_1(\aleph)$ we denote the space of absolutely summable families of scalars indexed by
$\Gamma$ with the sum norm.
A Banach space $X$ is said to be a $\mathscr{L}_{\infty, \lambda}$-space if every finite dimensional
subspace $F$ of $X$ is contained in another finite dimensional subspace of $X$ whose Banach-Mazur
distance to the corresponding $\ell_\infty^n$ is at most $\lambda$.
A $\mathscr{L}_\infty$-space is just a $\mathscr{L}_{\infty,\lambda}$-space for some $\lambda \geq 1$.
A Lindenstrauss space is a $\mathscr{L}_{\infty,1+}$-space, that is a Banach space which is a
$\mathscr{L}_{\infty,\lambda}$-space for all $\lambda>1$.

As usual, given a compact space $K$ we denote by $C(K)$ the Banach space of all real-valued
continuous functions on $K$, with the sup norm.
In this paper all topological spaces are assumed to be Hausdorff.
An $\mathcal M$-space is a Banach lattice where $\|x+y\|=\max (\|x\|,\|y\|)$ provided $x,y$
are disjoint.
Each $\mathcal{M}$-space can be represented as a sublattice of some $C(K)$-space.
Throughout the paper, {\sf ZFC} denotes the usual setting of set theory with the Axiom of Choice,
while {\sf CH} denotes the continuum hypothesis ($\aleph_1= 2^{\aleph_0} = \mathfrak c$) and
{\sf GCH} denotes the generalized continuum hypothesis (namely that $\aleph^+= 2^{\aleph}$ holds
for all infinite cardinals $\aleph$).

\subsection{Exact sequences}
A short exact sequence of Banach spaces is a diagram
\begin{equation}\label{SEX}
0 \To Y \stackrel{\imath}\To X \stackrel{\pi}\To Z \To 0
\end{equation}
where $Y$, $X$ and $Z$ are Banach spaces and the arrows are operators in such a way that
the kernel of each arrow coincides with the image of the preceding one.
By the open mapping theorem $\imath$ embeds $Y$ as a closed subspace of $X$ and $Z$ is
isomorphic to the quotient $X/\imath(Y)$.

The sequence (\ref{SEX}) is said to be trivial, or to split, if there is an operator
$p:X\to Y$ such that $p\imath={\bf 1}_Y$ (i.e., $\imath(Y)$ is complemented in $X$);
equivalently, there is an operator $s:Z\to X$ such that $\pi s={\bf 1}_Z$.
When properly classified and organized, the set of all possible exact sequences of the
form (\ref{SEX}) become a linear space, denoted by $\Ext(Z,Y)$, whose zero is the class
of trivial sequences; see  \cite{cabecastlong,castgonz} for explanations.
For this reason one often writes $\Ext(Z,Y)=0$ to indicate that every sequence of the
form (\ref{SEX}) is trivial.

A property $\mathscr P$ is said to be a \emph{3-space property} if $X$ has $\mathscr P$
whenever there is an exact sequence of the form (\ref{SEX}) in which both $Y$ and $Z$
have $\mathscr P$.

\subsection{The push-out and pull-back constructions}
A thorough description of the pull-back and push-out constructions in Banach spaces can be seen
in \cite{accgm1, accgm2,castgonz}.
Everything we need to know for this paper is that given an exact sequence (\ref{SEX}) and an
operator $t:Y\to B$ there is a commutative diagram
\begin{equation}\label{po-seq}
\begin{CD}
0  @>>>  Y @>\imath>> X @>\pi>> Z @>>>0 \\
  &  &   @Vt VV  @VVt'V
   @| \\
0  @>>> B @>\imath'>> \PO  @>>> Z@>>> 0
\end{CD}
\end{equation}
called the associated push-out diagram, in which  the lower row is an exact sequence which splits
if and only if $t$ extends to $X$, that is, there is an operator $T:X\to B$ such that $T\imath=t$.

Proceeding dually one obtains the associated pull-back sequence. Given an exact sequence (\ref{SEX})
and an operator $u:A\to Z$ there is a commutative diagram
\begin{equation}\label{pb-seq}
\begin{CD}
0  @>>>  Y @>\imath>> X @>\pi>> Z @>>>0 \\
  &  & @|  @A{'}\!u AA  @AA u A
    \\
0@>>> Y@>>> \PB  @>{'}\!\pi >> A@>>>0\\
\end{CD}
\end{equation}
whose lower sequence is  exact, and which shall be referred to as a pull-back diagram.
The splitting criterion is now as follows: the lower sequence splits if and only if $u$ lifts to $X$,
that is, there is an operator $U:A\to X$ such that $\pi U=u$.

\subsection{Filters}
Recall that a family $\mathscr F$ of subsets of a given set $I$ is said to be a filter if it is
closed under finite intersection, does not contain the empty set and one has $A\in\mathscr F$
provided $B\subset A$ and $A\in\mathscr F$.
An ultrafilter on $I$ is a filter which is maximal with respect to inclusion.
If $X$ is a (Hausdorff) topological space, $f:I\to X$ is a function, and $x\in X$, one says that
$f(i)$ converges to $x$ along $\mathscr F$ (written $x=\lim_\mathscr F f(i)$ to short) if
whenever $V$ is a neighborhood of $x$ in $X$ the set $f^{-1}(V)=\{i\in I: f(i)\in V\}$ belongs
to $\mathscr F$.
The obvious compactness argument shows that if $X$ is compact, and $\mathscr F$ is an ultrafilter
on $I$, then for every function $f:I\to X$ there is a unique $x\in X$ such that $x=\lim_\mathscr F f(i)$.

\subsection{The set-theoretic ultraproduct construction}\label{set}
It will be used in Section~\ref{ultra}. Let us recall some definitions, and fix notations.

Let $(S_i)_{i\in I}$ be a family of sets indexed by $I$ and let $\mathscr U$ be an ultrafilter on $I$.
The set-theoretic (or model-theoretic) ultraproduct $\langle S_i\rangle_\mathscr U$ is the product set
$\prod_iS_i$ factored by the equivalence $(s_i)=(t_i)\Leftrightarrow \{i\in I:s_i=t_i\}\in \mathscr U$.
The class of $(s_i)$ in $\langle S_i\rangle_\mathscr U$ is denoted $\langle (s_i)\rangle_\mathscr U$.
If we are given functions $f_i:S_i\to K$, where $K$ is some compact space, we can define another
function $f:\langle S_i\rangle_\mathscr U\to K$ by
$f(\langle s_i\rangle_\mathscr U)=\lim_{\mathscr U(i)}f_i(s_i)$.
Keisler's paper \cite{keisler} contains a good introduction to this topic and many related things.

\section{$\aleph$-injective Banach spaces}\label{aleph-inj}

In this section we extend some results proved in \cite{accgm1} for separably injective Banach spaces.
Recall that a Banach space $E$ is separably injective ($\aleph_1$-injective according to
Definition~\ref{def}) when $E$-valued operators extends to separable super-spaces, and that $E$ is
universally separably injective (universally $\aleph_1$-injective) when $E$-valued operators extend
from separable subspaces.
Our first result generalizes \cite[Proposition 3.2]{accgm1}.

\begin{prop}\label{alef-injective}
For a Banach space $E$ and a cardinal $\aleph$, the following assertions are equivalent:
\begin{enumerate}
\item $E$ is $\aleph$-injective. \item For every cardinal $\kappa<\aleph$, every operator from
a subspace of $\ell_1(\kappa)$ into $E$ extends to $\ell_1(\kappa)$.
\item For every Banach space $X$ and each subspace $Y$ such that $\dens(X/Y)<\aleph$, every
operator $t:Y\to E$ extends to $X$.
\item If $X$ is a Banach space containing $E$ and $\dens(X/E)<\aleph$, then $E$ is complemented in $X$.
\end{enumerate}
\end{prop}

\begin{proof}
It is clear that $(3)\Rightarrow (1) \Rightarrow (2)$.
To show that $(2)\Rightarrow (4)$, observe that if $\dens(X/E)<\aleph$ then there is a quotient map
$q: \ell_1(\kappa)\to X/E $ for some $\kappa<\aleph$.
The operator $q$ can be lifted to an operator $Q: \ell_1(\kappa)\to X$ whose restriction $Q_0$
to $\ker q$ actually takes values in $E$.
One therefore has a commutative diagram

$$
\begin{CD}
 0 @>>> \ker q @>i>> \ell_1(\kappa) @>q>> X/E @>>> 0\\
 & & @V{Q_0}VV @VQVV @| \\
 0 @>>> E @>>j> X @>>p> X/E @>>> 0.
\end{CD}
$$
By (2), there is a linear continuous extension $Q_1: \ell_1(\kappa) \to E$ of $Q_0$.
Since $(Q - jQ_1)i=0$, there is an operator $\nu:  X/E\to X$ such that $\nu q = Q - jQ_1$.
Since $p\nu q = q$, the expression $P ={\bf 1}_X - \nu p$ defines a projection on $X$ onto
the subspace $E$.
\smallskip

Now, to show that $(4) \Rightarrow (3)$ just form  the push-out diagram
$$
\begin{CD}
 0 @>>> Y @>\imath>> X @>\pi>> X/Y @>>> 0\\
 & & @V{t}VV @VV{t'}V @|\\
 0 @>>> E @>\imath'>> \PO @>>> \PO/E @>>> 0.
\end{CD}
$$
Since $\PO/E = X/Y$, the cardinality assumption is preserved and $E$ must be complemented
in $\PO$ by a projection $P$.
Thus, $Pt'$ yields an extension of $t$ as required.
\end{proof}

Our next result yields a homological characterization of $(2^\aleph)^+$-injectivity:

\begin{prop}\label{doselevado}
A Banach space $E$ is $(2^\aleph)^+$-injective if and only if it is complemented in
every superspace $W$ such that $W/E$ is a quotient of $\ell_\infty(\aleph)$.
\end{prop}

\begin{proof}
Every quotient of $\ell_\infty(\aleph)$ has density character at most $2^\aleph$;
so the necessity is clear by (4) in the preceding Proposition.

We prove now the sufficiency as follows: we will show that $E$-valued operators from
subspaces of $\ell_\infty(\aleph)$ can be extended to the whole $\ell_\infty(\aleph)$;
which in combination with the fact that $\ell_1(2^\aleph)$ is a subspace of
$\ell_\infty(\aleph)$ and Proposition \ref{alef-injective} (2) provides the result.
Thus, let $t: Z \to E$  be an operator defined on a subspace $Z$ of $\ell_\infty(\aleph)$.
One thus gets a push-out diagram
$$
\begin{CD}
0@>>> Z @>>> \ell_\infty(\aleph) @>>> \ell_\infty(\aleph)/Z  @>>> 0\\
&& @V{t}VV @VV{t'}V  @|\\
0@>>> E @>>> \PO @>>> \PO/Z  @>>> 0.
\end{CD}
$$
By hypothesis, $E$ is complemented in $\PO$ and thus there is a linear continuous projection
$P:\PO\to E$.
The operator $Pt': \ell_\infty(\aleph) \to E$ extends $t$.
Thus, it has been shown that every operator $t: Z\to E$ from a subspace of $\ell_\infty(\aleph)$
can be extended to the whole $\ell_\infty(\aleph)$.
To finish the proof it remains to prove the following result; probably it is known, but we
were unable to find an explicit reference:
\medskip

{\sc Claim}: $\ell_1(2^\aleph)$ is a subspace of $\ell_\infty(\aleph)$.

{\sc Proof of the Claim.}
The dual ball of $\ell_1(2^\aleph)$ in its weak$^*$-topology is homeomorphic to the product
$[-1,1]^{2^\aleph}$ which is a continuous image of $\{0,1\}^{2^\aleph}$.
This space has density $\aleph$, as we show now: observe that subsets of $2^\aleph$ can be
interpreted as elements of $\{0,1\}^{2^\aleph}$ via their characteristic functions.
Let us show that with this interpretation the clopen sets of $2^\aleph$ form a dense set of
$\{0,1\}^{2^\aleph}$: take a basic open set $U$; i.e., take points $p_1,\dots,p_n$ and
$q_1,\dots,q_m$ from $2^\aleph$ and form the basic open set
$$
U= \{x \in \{0,1\}^{2^\aleph}: \;\; x_{p_i} = 1 \quad \mathrm{and}\quad x_{q_i}=0\}.
$$
Find a clopen $C$ of $2^\aleph$ such that $p_1,\dots,p_n$ are in $C$, but $q_1,...,q_n$ do
not belong to $C$.
The characteristic function $\mathrm{1}_C\in U$.
Thus, since $2^\aleph$ has $\aleph$ many clopens, the dual ball of $\ell_1(2^\aleph)$ in
its weak$^*$-topology has density $\aleph$; and thus $\ell_1(2^\aleph)$ in its weak$^*$-topology
has density $\aleph$; and therefore $\ell_1(2^\aleph)$ can be embedded into $\ell_\infty(\aleph)$,
and the Claim is proved.
%
\end{proof}

The stability properties of the classes of (universally) $\aleph$-injective spaces are
gathered in the following Proposition (compare to \cite[Proposition 3.7]{accgm1}).

\begin{prop}\label{quotsicard}
Let $\aleph$ be an infinite cardinal.
\begin{enumerate}
\item The class of $\aleph$-injective spaces has the 3-space property.
\item The quotient of an $\aleph$-injective space by an $\aleph$-injective subspace is
$\aleph$-injective.
\item If $\varkappa \leq \aleph$, the quotient of a universally $\aleph$-injective space
by a $\varkappa$-injective subspace is universally  $\varkappa$-injective.
\end{enumerate}
\end{prop}
\begin{proof}
The proof of (1) follows from part (2) in Proposition \ref{alef-injective}:
let us consider an exact sequence $0 \To F \stackrel{j}\To E \stackrel{\pi}\To G \To 0$
in which both $F$ and $G$ are $\aleph$-injective.
Let $\phi: K \to E$ be an operator from a subspace $K$ of $\ell_1(\kappa)$ with $\kappa <\aleph$,
and let $\imath :K \to \ell_1(\kappa)$ denote the natural embedding; then $\pi\phi$ can be
extended to an operator $\Phi: \ell_1(\kappa) \to G$, which can in turn be lifted to an operator
$\Psi: \ell_1(\kappa) \to E$.
The difference $\phi - \Psi \imath$ takes values in $F$ and can thus be extended to an operator
$e: \ell_1(\kappa) \to F$.
The desired operator is $ \Psi + je$.

To prove (2) let us consider an exact sequence $ 0 \To F \To E \stackrel{\pi}\To G \To 0$ in
which both $F$ and $E$ are $\aleph$-injective.
Let $\phi: Y \to G$ be an operator from a subspace $Y$ of a space $X$ with
$\mathrm{dens}\; X < \aleph$.
Consider the pull-back diagram
$$
\begin{CD} 0@>>> F @>>> E @>\pi>> G @>>> 0\\
& & @| @AA{\Phi}A @AA{\phi}A\\
 0@>>> F@>>> \PB @>Q>>
Y@>>> 0\end{CD}
$$
and observe that since $F$ is $\aleph$-injective, the lower exact sequence splits, so $Q$
admits a linear continuous selection $s:Y\to \PB$.
By the $\aleph$-injectivity of $E$, there exists an operator $T:X\to E$ agreeing with
$\Phi s$ on $Y$.
Then $\pi T:X\to G$ is the desired extension of $\phi$ since
$\pi T|_Y = \pi\Phi s =\phi Q s = \phi$.

To prove (3), assume that $E$ is universally $\aleph$-injective and $F$ is $\varkappa$-injective.
The previous proof, reproduced verbatim, shows that $G$ is universally $\varkappa$-injective.
\end{proof}

It is perhaps worth to remark that an abstract homological proof that all properties having
the form $\Ext(X, -)=0$ are 3-space properties can be found in \cite{cabecastlong}.
The connection with Proposition \ref{quotsicard} (1) is that Proposition \ref{alef-injective} (4)
can be read in this language: a Banach space $E$ is $\aleph$-injective if and only if it verifies
$\Ext(F, E)=0$ for all Banach spaces $F$ with $\mathrm{dens}\; F<\aleph$.
 \smallskip

It would be interesting to know whether the class of universally $\aleph$-injective spaces enjoys
the 3-space property.
This was shown for $\aleph=\aleph_1$ in  \cite{accgm1}, but the proof is based on the equivalence
with the property of $\ell_\infty$-super-saturation: every separable subspace of $E$ is contained
in a copy of $\ell_\infty$ contained in $E$; see \cite[Proposition 5.2]{accgm1}.
Apparently there is no higher cardinal analogue for such property.
Indeed, the obvious extension fails because there exist injective Banach spaces with arbitrarily
large density character, like the spaces $L_\infty(\mu)$ for finite $\mu$, that do not contain
subspaces isomorphic to $\ell_\infty(\aleph_1)$ --this is so since a family of mutually disjoint
sets of positive measure on a finite measure space must be countable.
We can obtain a partial analogue introducing the following concept.

\begin{defin}
Let $\aleph$ be an infinite cardinal.
We say that a subspace $Y$ of a Banach space $X$ if \emph{$c_0(\aleph)$-supplemented} if there
exists another subspace $Z$ of $X$ isomorphic to $c_0(\aleph)$ such that $Y\cap Z=0$ and the sum
$Y+Z$ is closed.
In this case we will also say that $Z$ is a \emph{$c_0(\aleph)$-supplement} of $Y$.
\end{defin}

\begin{lema}\label{c_0-suppl}
Each subspace of $\ell_\infty(\aleph)$ with density character $\leq\aleph$  is
$c_0(\aleph)$-supplemented.
\end{lema}
\begin{proof}
Let $I$ have cardinality $\aleph$ and let $\{I_j : j\in J\}$ be a family of disjoint subsets of $I$
with $|I_j|=\aleph$ for every $j$ and $|J|=\aleph$.
Let $Y$ be a subspace of $\ell_\infty(I)$ with density character $\leq\aleph$.
Since $\dens\big(\ell_\infty(I_j)\big) > \aleph$, for each $j\in J$ we can find
$x_j\in \ell_\infty(I_j)$ with $\|x_j\|=1$ and ${\rm dist}(x_j,Y)>1/2$.
In this way we obtain a family $\{x_j : j\in J\}$ in $\ell_\infty(I)$ isometrically equivalent
to the basis of $c_0(I)$. Let $\pi:\ell_\infty(I)\to \ell_\infty(I)/Y$ denote the quotient map.
Since $\inf\{\|\pi(x_j)\| : j\in J\} \geq 1/2 >0$, by \cite[Theorem 3.4]{disjoint} there exists
$J_1 \subset J$ with $|J_1|=|J|$ such that the restriction of $\pi$ to the closed subspace
generated by $\{x_j : j\in J\}$ is an isomorphism.
That  space is a $c_0(\aleph)$-supplement of $Y$.
\end{proof}

We thus get the partial extension result announced above.

\begin{teor}\label{aleph-Thetastar=aleph-Sinfty}
Let $X$ be a universally $\aleph^+$-injective Banach space and let $Y$ be a
$c_0(\aleph)$-supplemented subspace of $X$ with $\dens(Y)\leq \aleph$.
Then $Y$ is contained in a subspace of $X$ isomorphic to $\ell_\infty(\aleph)$.
\end{teor}
\begin{proof}
Let $Y_0$ be a subspace of $\ell_\infty(\aleph)$ isomorphic to $Y$ and let $t:Y_0\to Y$ be a
(bijective) isomorphism with $\|t^{-1}\|=1$.
By Lemma \ref{c_0-suppl}, $Y_0$ is $c_0(\aleph)$-supplemented in $\ell_\infty(\aleph)$.
We can find projections $P$ on $X$ and $Q$ on $\ell_\infty(\aleph)$ such that
$Y\subset \ker P$, $Y_0\subset \ker Q$, and both ranges $\ran P$ and $\ran Q$ are isomorphic
to $\ell_\infty(\aleph)$.

Indeed, let $\pi:X\to X/Y$ be the quotient map.
Let $L_0$ and $L$ denote $c_0(\aleph)$-supplements for $Y_0$ and $Y$.
The isomorphism $t:Y_0\to Y$ admits an extension $\hat t:Y_0\oplus L_0\to Y\oplus L$,
and the universal $\aleph^+$-injectivity of $X$ allows us to find an operator
$I:\ell_\infty(\aleph)\to X$ extending $\hat t$, hence $\pi I$ is an isomorphism on a copy
$L_0$ of $c_0(\aleph)$.
A classical result of Rosenthal \cite[Theorem 1.3]{disjoint} yields that $\pi I$ is also
an isomorphism on a copy $M_0$ of $\ell_\infty(\aleph)$ inside $\ell_\infty(\aleph)$.
Therefore $M = I(M_0)$ is a subspace of $X$ isomorphic to $\ell_\infty(\aleph)$ where the
restriction of $\pi$ is an isomorphism.
Now $X/Y = \pi(M)\oplus N$, with $N$ a closed subspace.
Hence $X = M \oplus \pi^{-1}(N)$, and it is enough to take as $P$ the projection with range $M$
and kernel $\pi^{-1}(N)$.
Similarly, the quotient map $\pi_0:\ell_\infty(\aleph)\to \ell_\infty(\aleph)/Y_0$ is an
isomorphism on $L_0$, and Rosenthal's result implies that $\pi_0$ is an isomorphism on a
subspace $M_0$ of $\ell_\infty(\aleph)$ isomorphic to $\ell_\infty(\aleph)$.
Thus $\ell_\infty(\aleph)/Y_0 = \pi_0(M_0)\oplus N_0$, with $N_0$ a closed subspace.
Hence $\ell_\infty(\aleph) = M_0\oplus \pi_0^{-1}(N_0)$, and it is enough to take as $Q$
the projection with range $M_0$ and kernel $\pi_0^{-1}(N_0)$.

Since $\ker P$ and $\ker Q$ are universally $\aleph$-injective spaces, there are operators
$U:X\to \ker Q$ and $V:\ell_\infty(\aleph)\to \ker P$ such that $V|_{Y_0}= t$ and $U|_Y= t^{-1}$.
Note that $\|U\|\geq 1$.
Let $W:\ell_\infty(\aleph) \to \ran P$ be an operator satisfying $\|Wx\|\geq\|x\|$ for all $x$.
We will show that the operator
$$
T = V + W({\bf 1}_{\ell_\infty(\aleph)}-UV): \ell_\infty(\aleph) \To X
$$
is an isomorphism (into).
Since $\ran V\subset \ker P$ and $\ran W\subset \ran P$, there exists $C>0$ such that
$$
\|Tx\| \geq C \max\{ \|Vx\|, \|W({\bf 1}_{\ell_\infty(\aleph)}-UV)x\|\}
$$
for every $x\in \ell_\infty(\aleph)$.
Now, if $\|Vx\|< (2\|U\|)^{-1}\|x\|$, then $\|UVx\| < (1/2)\|x\|$; hence
$$
\|W({\bf 1}_{\ell_\infty(\aleph)}-UV)x\|\geq \|({\bf 1}_{\ell_\infty(\aleph)}-UV)x\| >\|x\|/2.
$$
Thus $\|Tx\| \geq C (2\|U\|)^{-1}\|x\|$ for every $x\in X$; hence $Y$ is contained in the
range of $T$, which is isomorphic to $\ell_\infty(\aleph)$.
\end{proof}

By the Lindenstrauss-Rosenthal theorem \cite{lindrose} any isomorphism between two separable
subspaces of $\ell_\infty$ can be extended to an automorphism of $\ell_\infty$.
As a consequence of Theorem \ref{aleph-Thetastar=aleph-Sinfty} we can prove that
universally $\aleph^+$-injective spaces enjoy a similar property.

\begin{teor}\label{aleph-Thetastar=aleph-automorphic}
Let $X$ be a universally $\aleph^+$-injective Banach space, and let $Y_1$ and $Y_2$ be
isomorphic $c_0(\aleph)$-supplemented subspaces of $X$ with $\dens(Y_i)\leq \aleph$.
Then every isomorphism from $Y_1$ onto $Y_2$ extends to an automorphism of $X$.
\end{teor}
\begin{proof}
Note that we can modify the proof of Theorem \ref{aleph-Thetastar=aleph-Sinfty} in such a
way that the subspace $Z$ isomorphic to $\ell_\infty(\aleph)$ that contains $Y$ has a
complement isomorphic to $X$.
Indeed, if we write $\ran(P)$ as the direct sum of two copies of $\ell_\infty(\aleph)$ and
take $W$ so that its image is contained in one of the summands, then the complement $Z'$
of $Z$ in $X$ contains a subspace isomorphic to $\ell_\infty(\aleph)$; hence
$$
Z' \sim Z'' \oplus \ell_\infty(\aleph)
    \sim Z'' \oplus \ell_\infty(\aleph)\oplus \ell_\infty(\aleph)
    \sim Z' \oplus \ell_\infty(\aleph)
    \sim Z'\oplus Z  \sim X.
$$
So, for each $i=1,2$, we can assume that $Y_i$ is contained in a subspace $Z_i$ isomorphic
to $\ell_\infty(\aleph)$ such that the complement of $Z_i$ in $X$ is isomorphic to $X$.
Therefore, given an isomorphism $T:Y_1\to Y_2$, since the quotients $Z_1/Y_1$ and $Z_2/Y_2$
are not reflexive, we first extend $T$ to an isomorphism $\tau$ from $Z_1$ onto $Z_2$,
which clearly can be extended to an automorphism of $X$.

The extension from $T:Y_1\to Y_2$ to $\tau :Z_1\to Z_2$ can be obtained as in the
proof of part (i) of \cite[Theorem 2.f.12]{lindtzaf}).
First, as in the proof of Theorem \ref{aleph-Thetastar=aleph-Sinfty}, we get projections
$P_i$ on $Z_i$ ($i=1,2$) with $Y_i \subset \ker P_i$ and both $\ker P_i$ and $\ran P_i$
isomorphic to $\ell_\infty(\aleph)$.

Since the space $\ran({\bf 1}_{Z_2}-P_2)$ is injective and contains $Y_2$, there exists
an extension $S_1:Z_1\to \ran({\bf 1}_{Z_2}-P_2)$ of $T$, and similarly there is an
extension $S_2:Z_2\to \ran({\bf 1}_{Z_1}-P_1)$ of $T^{-1}$.
Let $R$ be an isomorphism from $Z_1$ onto $\ran(P_2)$ with $\|R^{-1}\|>1$, and define
$\hat T:Z_1\to Z_2$ by $\hat T= S_1 + R({\bf 1}_{Z_1}-S_2S_1)$.
Since $({\bf 1}_{Z_1}-S_2S_1)|_{Y_1}=0$ it follows that $\hat T$ is an extension of $T$,
and as in the proof of part (i) of \cite[Theorem 2.f.12]{lindtzaf}) we can check that
$\hat T$ is an into isomorphism.

Note that the subspace $\hat T\big(\ran({\bf 1}_{Z_1}-P_1)\big)$ is complemented in
$Z_2$ with complement $W$ isomorphic to $Z_2$.
Let $R_0$ be an isomorphism from $\ran P_1$ onto $W$.
The operator $\tau=R_0P_1 + \hat T({\bf 1}_{Z_1}-P_1)$ is an isomorphism from $Z_1$
onto $Z_2$ extending $T$.
\end{proof}

Further differences between $\aleph$-injectivity and separable injectivity is that
Sobczyk's theorem has no simple counterpart for higher cardinals: indeed, $c_0(\aleph)$
is never $\aleph_2$-injective just because its complemented subspace $c_0$ is not:
recall from \cite{accgm1,castgonz} the existence of the Johnson-Lindenstrauss nontrivial exact sequences
$$
\begin{CD} 0@>>> c_0@>>> {J\!L} @>>> c_0(\aleph_1)@>>>0.\end{CD}
$$
Perhaps the role of $c_0$ could be played by Hasanov's ``filter version'' of $c_0$
(see \cite{hasa}).
Recall that a filter $\mathscr F$ on a set $S$ is called $\aleph$-complete if whenever
$A_i \in \mathscr F$ for all $i\in I$ with  $|I|< \aleph$ then $\bigcap_{i \in I} A_i$
is again in $\mathscr F$.
The space $c_0^\mathscr F(S)$ is the closed linear span in $\ell_\infty(S)$ of the set
$\{x\in \ell_\infty(S) : \lim_{\mathscr F} x = 0 \} $.
Hasanov shows in \cite{hasa} that if $\mathscr F$ is $\aleph$-complete, then
$c_0^\mathscr F(S)$ is at most 2-complemented in any superspace $E$ such that
$\dens (E/c_0^\mathscr F(S))\leq \aleph$.
Thus, it is $(2,\aleph^+)$-injective.

\section{$(1,\aleph)$-injective Banach spaces}\label{1aleph-inj}

The $(1,\aleph)$-injective spaces can be characterized as follows (this result can be
essentially found as the remark after Corollary 2, p. 56, in \cite{lind}; the ``if part''
is due to Aronszajn and Panitchpakdi, see \cite[Theorem 3]{a-p}).

\begin{lema}\label{crit}
A Banach space $E$ is $(1,\aleph)$-injective if and only if every family of less than
$\aleph$ mutually intersecting balls of $E$ has nonempty intersection.
\end{lema}

\begin{proof}
{\sc Sufficiency.} Take an operator $t:Y\to E$, where $Y$ is a closed subspace of $X$,
where $\dens X<\aleph$. We may and do assume $\|t\|=1$.
Let $z\in X\backslash Y$ and let $Y_0$ be a dense subset of $Y$ forming a linear space
over the rational numbers with $|Y_0|<\aleph$ and, for each $y\in Y_0$, consider the
ball $B(ty,\|y-z\|)$ in $E$.
Any two of these balls intersect, since for $y_1,y_2\in Y_0$ we have
$$
\|ty_2-ty_1\|\leq \|t\|\|y_2-y_1\|\leq \|y_2-z\|+\|y_1-z\|.
$$
The hypothesis is that there is
$$
f\in\bigcap_{y\in Y_0} B(ty,\|y-z\|)=\bigcap_{y\in Y} B(ty,\|y-z\|).
$$
It is clear that the map $T:Y+\langle z\rangle \to E$ defined by $T(y+cz)=ty+cf$ is
an extension of $t$ with $\|T\|=1$.
The rest is clear: use Zorn lemma.\\

{\sc Necessity.} We begin with the observation that if two closed balls of a Banach
space have a common point, then the distance between the centers is at most the sum
of the radii.
In $\ell_\infty(\aleph)$ that necessary condition is sufficient.
On the other hand, in $\ell_\infty(\aleph)$, every family of mutually intersecting
balls has nonempty intersection ---this is trivial.
So, let $E$ be $(1,\aleph)$-injective and suppose $B(e_i,r_i)$ is a family of less
than $\aleph$ mutually intersecting balls in $E$.
Let $Y$ be the closed subspace of $E$ spanned by the centers, so that $\dens Y<\aleph$.
Let $\jmath: Y\to\ell_\infty(\aleph)$ be any isometric embedding.
Notice that even if $B_Y(e_i,r_i)=B(e_i,r_i)\cap Y$ need not be mutually intersecting
in $Y$, any two balls of the family $B(\jmath(e_i), r_i)$ meet in $\ell_\infty(\aleph)$
because the distance between the centers does not exceed the sum of the radii.
Therefore the intersection
$$
\bigcap_i B(\jmath(e_i), r_i)
$$
contains some point, say $x\in\ell_\infty(\aleph)$.
Let $X$ be the subspace spanned by $x$ and $\jmath(Y)$ in $\ell_\infty(\aleph)$.
The hypothesis on $E$ allows one to extend the inclusion $\imath: Y\to E$ through
$\jmath: Y\to X$ to an operator $I:X\to E$ without increasing the norm; i.e.,
$I\jmath =\imath$.
Since $\|I(x) - e_i\| = \|I(x- \jmath e_i)\| \leq \|x- \jmath e_i\|$ one gets
$I(x)\in \bigcap_iB(e_i,r_i)$.
\end{proof}

Proposition 6.2 of \cite{accgm1} asserts that, under {\sf CH}, $1$-separably injective
spaces are universally $1$-separably  injective.
This admits a higher cardinal counterpart, which stems from remark 6, p. 223,
in \cite{lindpams}.

\begin{prop}\label{1-aleph-GCH=univ}
Under {\rm {\sf GCH}}, every $(1, \aleph)$-injective Banach space is universally
$(1,\aleph)$-injective.
\end{prop}
\begin{proof}
Let $E$ be an $(1, \aleph)$-injective Banach space and let $Y$ be a density character
$\aleph$ subspace of a space $X$ and let $t:Y\to E$ be an operator.
Let  $j: Y\to \ell_\infty(\aleph)$ be an isometric embedding; and observe that,
under {\sf GCH}, the space $\ell_\infty(\aleph)$ has density character $\aleph^+$.
Since a set of cardinal $\aleph^+$ can be written as the union of an increasing chain
of sets of cardinal $\aleph$, write $\ell_\infty(\aleph)$ as the union of an increasing
chain of subspaces with density character $\aleph$.
There is no loss of generality in assuming that the first set of the chain is $Y$.
Use transfinite induction and the $(1, \aleph)$-injectivity of $E$ to extend $t$ to
an operator $T: \ell_\infty(\aleph) \to E$ with the same norm
(see \cite[Lemma 6.1]{accgm1} for details).
Extend $j$ to an operator $J:X\to \ell_\infty(\aleph)$ with the same norm.
The composition $TJ$ is the desired equal norm extension of $t$.
\end{proof}

\section{Spaces of continuous functions}\label{C(K)-spaces}

The following omnibus result summarizes what is known about the
interplay between the $(1,\aleph)$-injectivity of $C(K)$, the
topological properties of $K$ and the lattice structure of $C(K)$.
We will need a simple Lemma which can be found in \cite[lemma 6.4]{accgm1}.

\begin{lema}\label{positiveextension}
Let $K,L,M$ be compact spaces and let $f:K\To M$ be a continuous
map, with \mbox{$\jmath=f^{\circ}:C(M)\To C(K)$} its induced
operator, and let $\imath: C(M)\To C(L)$ be a positive norm one
operator. Suppose that $S:C(L)\To C(K)$ is an operator with
$\|S\|=1$ and $S\imath = \jmath$. Then $S$ is a positive operator.
\end{lema}

Recall that a cozero set in a topological space $K$ is an open set of the
form $\{x\in K: f(x)\neq 0\}$ for some $f\in C(K)$. One has:

\begin{teor}\label{people} For a compact space $K$ and a
cardinal number $\aleph$, the following  statements are
equivalent:
\begin{itemize}
\item[(a)] $C(K)$ is $(1,\aleph)$-injective. \item[(b)] Given
subsets $L$ and $U$ of $C(K)$ with $|L|,|U|<\aleph$ such that
$f\leq g$ for every $f\in L$ and $g\in U$, there exists $h\in
C(K)$ separating them, that is, such that  $f\leq h\leq g$ for all
$f\in L$ and $g\in U$. \item[(c)] Every family of mutually
intersecting balls in $C(K)$ of cardinal less than $\aleph$ has
nonempty intersection. \item[(d)] Every couple of disjoint open
sets $G$ and $H$ of $K$ which are the union of less than $\aleph$
many closed sets have disjoint closures. \item[(e)] Every couple
of disjoint open sets $G$ and $H$ of $K$ which are the union of
less than $\aleph$ many cozero sets have disjoint closures.
\end{itemize}
\end{teor}
\begin{proof}
We first prove the implications $(a)\Rightarrow (b) \Rightarrow
(c)\Rightarrow (a)$ in that order. Let $L$ and $U$ be as in (b).
We  consider $C(K)$ as a subalgebra of $\ell_\infty(K)$. Let
$\eta\in\ell_\infty(K)$ such that $f\leq\eta\leq g$ for all $f\in
L$ and $g\in U$. Let $A$ be the least unital closed subalgebra of
$\ell_\infty(K)$ containing $L, U$ and $\eta$, and let $B=A\cap
C(K)$. Clearly, $\dens A< \aleph$. By (a), the inclusion of $B$
into $C(K)$ extends to a norm-one operator $I:A\to C(K)$. Let $L$
be the maximal ideal space of $A$ and $M$ that of $B$. By general
representation theorems we have $A=C(L), B=C(M)$ (see for
instance \cite[Theorem 4.2.5]{a-k}) and a commutative diagram
$$
\setlength{\unitlength}{0.1mm}
\begin{picture}(1300,350)(-160,-100)
\linethickness{0.4pt}

\put(330,130){\vector(1,-1){150}}
\put(740,130){\vector (-1,-1){150}}
\put(400,160){\vector(1,0){340}}


\put(250,150){$C(M)$}

\put(780,150){$C(L)$}

\put(480,-80){$C(K)$}


\put(550,170){{\footnotesize $\imath$}}

\put(360,25){\footnotesize{ $\jmath$}} \put(700,25){\footnotesize{$I$}}

\end{picture}
$$

By Lemma \ref{positiveextension}, $I$ is a positive operator,
hence $I\eta$ separates $L$ from $U$.

We check now $(b) \Rightarrow (c)$. Let $(B_i)_{i\in I}$ be a
family of mutually intersecting balls, where $|I|< \aleph$.
Writing $B_i=B(f_i,r_i)$, we have $ \|f_i-f_j\|\leq r_i+r_j$ for
all $i,j\in I$, that is,
$$
f_i-r_i\leq f_j+r_j\quad\quad(i,j\in I).
$$
By (b) there is $h\in C(K)$ such that
$$
f_i-r_i\leq h\leq f_j+r_j\quad\quad(i,j\in I).
$$
In particular $f_i-r_i\leq h\leq f_i+r_i$, that is, $h\in\bigcap_i
B_i$. The implication $(c) \Rightarrow (a)$ is contained in
Lemma~\ref{crit}.\\

We pass to the string $(b)\Rightarrow (d)\Rightarrow
(e)\Rightarrow (b)$. Assume that $(b)$ holds and let $G$ and $H$
be as in $(d)$, so that $G = \bigcup_{\alpha\in I}C_\alpha$ and
$H=\bigcup_{\alpha\in I}D_\alpha$, where $C_\alpha$ and $D_\alpha$
are closed subsets of $K$ and $|I|<\aleph$. For every
$\alpha \in I$, let $f_\alpha\in C(K)$, $0\leq f_\alpha\leq 1$,
such that $f_\alpha|_{K\setminus G}=0$ and
$f_\alpha|_{C_\alpha}=1$, and let $g_\alpha\in C(K)$, $0\leq
g_\alpha\leq 1$, such that $g_\alpha|_{K\setminus H}=1$ and
$g_\alpha|_{D_\alpha}=0$. The sets $L=\{f_\alpha : \alpha\in I\}$
and $U= \{g_\alpha : \alpha\in I\}$ satisfy the assumptions of
condition $(b)$. The function $h\in C(K)$ that separates $L$ and
$U$ has the property that $h|_G=1$ and $h|_H=0$, hence
$\overline{G}\cap\overline{H}=\varnothing$. That $(d)$ implies
$(e)$ is a consequence of the fact that each cozero set is the
union of countably many closed sets, namely for $f\in C(K)$,
$$
\{x\in K : f(x)\neq 0\} = \bigcup_{n\in\mathbb{N}}\{x\in K :
|f(x)|\geq 1/n\}.
$$
Assume now that $(e)$ holds. As a first step towards $(b)$, we
prove it modulo a given positive $\varepsilon$.\\

{\sc Claim.} Given $U$ and $L$ like in $(b)$ and given
$\varepsilon>0$, there exists $h\in C(K)$ such that
$f-\varepsilon\leq h \leq g+\varepsilon$ for every $f\in L$ and
$g\in U$.\\

{\sc Proof of the
claim.} By homogeneity, it is enough to consider the case
$\varepsilon=1$. Let $N\in\mathbb{N}$ be such that $-N<f_0\leq
g_0<N$ for some $f_0\in L$ and $g_0\in U$. Let
$I=\{n\in\mathbb{N} : -N<n<N\}$. For every $n\in I$, let
\begin{eqnarray*}
G_n &=& \left\{x\in K : f(x)>n \text{ for some } f\in L \right\} = \bigcup_{f\in L} f^{-1}(n,+\infty),\\
H_n &=& \left\{x\in K : g(x)<n \text{ for some } g\in U \right\} = \bigcup_{g\in U} g^{-1}(-\infty, n).
\end{eqnarray*}
 For each $n$,
$G_n$ and $H_n$ are disjoint open sets which are the union of less
than $\aleph$ cozero sets, because $|L|,|U|<\aleph$ and $f^{-1}(n,
+\infty)$ is itself a cozero set (it is the complement of the zero
set of $\max\{f-n,0\}$). Hence
$\overline{G_n}\cap\overline{H_n}=\varnothing$, therefore there
exists $h_n\in C(K)$, $-1 \leq h_n\leq 1$ such that $h_n|_{G_n} =
1$ and $h_n|_{H_n} = -1$. We shall check that
$h=\frac{1}{2}\sum_{n\in I}h_n\in C(K)$ is the desired function.
For $f\in L$ and $x\in K$,
\begin{eqnarray*}
h(x) &=& \frac{1}{2}  \sum_{n\in I}h_n(x)  = \frac{1}{2} \left(
\sum_{n\in I, n<f(x)}(1) + \sum_{n\in I, n\geq
f(x)}h_n(x)\right)\\
&\geq& \frac{|\{n\in I, n<f(x)\}|-
|\{n\in I, n\geq f(x)\}|}{2}\\ &\geq& f(x)-1.
\end{eqnarray*}
Similarly, one gets that $h(x)\leq g(x) +1$ for all $g\in U$ and
$x\in K$. {\sc End of the proof of the claim.}\\

Now, if $U$ and $L$ are sets like in $(b)$ we construct
inductively a sequence of new sets $U_n,L_n\subset C(K)$
and functions $h_n\in C(K)$ as follows: $L_0=L$, $U_0=U$;
$h_n\in C(K)$ is such that $f-2^{-n}\leq h_n\leq g+2^{-n}$
for all $f\in L_n$, $g\in U_n$;
$L_{n+1} = L_n\cup\{h_n-2^{-n}\}$,
$U_{n+1} = U_n\cup \{h_n+2^{-n}\}$.
This can be performed because of the preceding claim.
Notice that the sequence $(h_n)_{n\in\mathbb{N}}$ is
uniformly convergent because for $m<n$, $h_{m}-2^{-m}\in L_n$,
$h_{m}+2^{-m}\in U_n$, hence
$$
h_{m}-2^{-m}-2^{-n}\leq h_n\leq h_m + 2^{-m} +2^{-n}
\Rightarrow  \|h_n - h_m\| \leq 2^{-m+1}.
$$
We can consider thus $h=\lim_n h_n$.
This function belongs to $C(K)$ and  satisfies $f\leq h \leq g$ for $f\in L$ and $g\in U$.
\end{proof}

\begin{remar}
The preceding Theorem summarizes or generalizes many earlier results.
The equivalence between (a), (d) and (e) can be traced back to \cite[Theorem 2]{a-p}
although Aronszajn and Panitchpakdi manage a condition intermediate between (d) and (e);
see also Henriksen's note \cite{henriksen}.
Neville's \cite[Theorem 2]{neville} is clearly related to the equivalence between (a) and (b).
The equivalence between (b) and (e) when $\aleph=\aleph_1$ if due to Seever \cite[Theorem 2.5]{seever}.
\end{remar}

As a consequence of Theorem~\ref{people} we get (cf. \cite[Section 4.3]{accgm1}):

\begin{prop}
The space $C(\N^*) = \ell_\infty/c_0$ is $(1,\aleph_1)$-injective.
\end{prop}

Recall that $\N^* = \beta \N\setminus \N$. We show now that no
cardinal improvement is possible.

\begin{prop}\label{not1aleph2}
$C(\N^*)$ is not $(1,\aleph_2)$-injective.
\end{prop}
\begin{proof} A classical construction in set theory known as the Hausdorff
gap \cite{hausdorff} yields the existence of two
$\omega_1$-sequences of clopen sets in $\N^*$, say $(a_i)$ and
$(b_i)$ where $ i \in \omega_1$, such that $(a_i)$ is increasing,
$(b_i)$ is decreasing, $a_i \subset b_j$ for all $i,j$ and with
the additional property that for no clopen set $c$ one may have
$a_i \subseteq c \subseteq b_j$ for all $i,j\in\omega_1$.
Considering the characteristic functions of those clopen sets,
condition $(b)$ of Theorem \ref{people} is violated for
$\aleph=\aleph_2$ (take into account that zero-dimensional
compacta are in fact strongly zero-dimensional, that is, disjoint
zero sets can be put into disjoint clopen sets).
\end{proof}

If we deal with $\aleph$-injectivity instead of
$(1,\aleph)$-injectivity, the matter becomes more complicated:
since $C(\N^*)$ contains an uncomplemented copy of itself
\cite{castplic} it is not $\mathfrak c^+$-injective. We do
not know whether it is consistent that $C(\mathbb{N}^\ast)$ is
$\aleph_2$-injective.

\section{Ultraproducts}\label{ultra}

Perfect examples of compact spaces as those of Theorem
\ref{people} -- the object we will study in the next section -- can
be obtained via ultraproducts. Let us briefly recall the definition and some basic properties of
ultraproducts of Banach spaces. For a detailed study of this
construction at the elementary level needed here we refer the
reader to Heinrich's survey paper \cite{heinrich} or Sims' notes
\cite{sims}.

Let $\Xii$ be a family of Banach spaces indexed by $I$ and let $\mathscr U$ be an ultrafilter on $I$.
The space  $ \ell_\infty(X_i)$ endowed
with the supremum norm is a Banach space, and $ c_0^\U(X_i)=
\{(x_i) \in \ell_\infty(X_i) : \lim_{\U(i)} \|x_i\|=0\} $ is a
closed subspace of $\ell_\infty(X_i)$. The ultraproduct of the
spaces $\Xii$ following $\U$ is defined as the quotient
$$
[X_i]_\U = {\ell_\infty(X_i)}/{c_0^\U(X_i)},
$$
with the quotient norm.
We denote by $[(x_i)]$ the element of $[X_i]_\U$ which has the
family $(x_i)$ as a representative. It is not difficult to show
that $ \|[(x_i)]\| = \lim_{\U(i)} \|x_i\|. $ In the case $X_i = X$
for all $i$, we denote the ultraproduct by $X_\U$, and call it the
ultrapower of $X$ following $\U$.

There is an obvious connection between $\XiU$ and the set-theoretic ultraproduct
$\langle X_i\rangle_U$ defined in Subsection~\ref{set}: indeed, the former space
can be obtained from the latter, first taking the elements for which the seminorm
$$
\langle(x_i)\rangle_\U\longmapsto \lim_{\U(i)}\|x_i\|
$$
is finite (we may consider the original norms on the $X_i$ as taking values on the
extended ray $[0,\infty]$), and then taking quotient by the kernel of the seminorm.


If $\Xii$ is a family of Banach algebras, then $\ell_\infty(X_i)$
is also a Banach algebra, with the coordinatewise product. Thus,
if $\U$ is an ultrafilter on $I$, $c_0^\U(X_i)$ is an ideal in
$\ell_\infty(X_i)$ and $[X_i]_\U$ becomes a Banach algebra with
product
$$
[(x_i)] \cdot [(y_i)] = [(x_i\cdot y_i)].
$$

Thus, if $\Kii$ is a family of compact spaces, the algebra
$[C(K_i)]_\U$ is isometrically isomorphic to $C(K)$, for
some compact space $K$; see \cite[Theorem 4.2.5]{a-k}.
This compact is called the (topological) ultracoproduct of $\Kii$, and it is
denoted $\KiU$; actually $\KiU$ is the maximal ideal space of $[C(K_i)]_\U$ equipped
with the relative weak* topology.
We refer the reader to \cite[Section 5]{bankston-survey} for a purely topological
description of the ultracoproduct construction, although we will not use it.

An ultrafilter $\U$ on a set $I$ is countably incomplete if here
is a decreasing sequence $(I_n)$ of subsets of $I$ such that
$I_n\in \U$ for all $n$, and $\bigcap_{n=1}^\infty
I_n=\varnothing$.

Notice that $\U$ is countably incomplete if
and only if there is a function $n:I\to \N$ such that
$n(i)\to\infty$ along $\U$ (equivalently, there is a family
$\e(i)$ of strictly positive numbers converging to zero along
$\U$). It is obvious that any countably incomplete ultrafilter is
non-principal  and also that every non-principal (or free)
ultrafilter on $\N$ is countably incomplete.

In order to present the main result of the Section we need Keisler's notion of an $\aleph$-good
ultrafilter \cite{chang-keisler, bankston-survey}.

\begin{defin}
Let $\fin(S)$ denote the set of finite
subsets of a given set $S$. If $\mathscr U$ is an ultrafilter on $I$, we say
that $f:\fin(S)\to \U$ is monotone
(respectively, multiplicative) if $f(A) \supset f(B)$
whenever $A\subset B$ (respectively, if $f(A\cup B)=f(A)\cap
f(B)$). The ultrafilter $\mathscr U$ is said to be
$\aleph$-good if, for every $S$ with $|S|< \aleph$, and
every monotone $f:\fin(S)\to \U$, there is a
multiplicative $g:\fin(S)\to \U$ such that
$g(A)\subset f(A)$ for all $A$.
\end{defin}

Every set of cardinality $\aleph$ supports $\aleph^+$-good ultrafilters
(see \cite[Theorem 10.4]{comfort-negrepontis} or \cite[Theorem 6.1.4]{chang-keisler}). Hence, every countably incomplete ultrafilter is $\aleph_1$-good. Since an  $\aleph^{++}$-good ultrafilter on set of cardinality $\aleph$ is necessarily fixed
(by saturation and \cite[Proposition 4.2.2]{chang-keisler}), $\aleph^+$-good ultrafilters will be simply called ``good ultrafilters''.

\begin{teor}\label{lind2}
Let $\U$ be a countably incomplete, $\aleph$-good ultrafilter on
$I$ and let $X_i$ be  a family of Banach spaces indexed by $I$. If
$[X_i]_\U$ is a Lindenstrauss space, then it is
$(1,\aleph)$-injective.
\end{teor}

\begin{proof}
The key point is the saturation property of the set-theoretic
ultraproducts via good ultrafilters. Let $(S_i)_{i\in I}$ be a
family of sets and let $\U$ be an ultrafilter on $I$. A subset $A$
of $\langle S_i\rangle_\U$ is called internal if there are sets
$A_i\subset S_i$ such that $A=\langle A_i\rangle_\U$. It can be
proved (see \cite[Theorem 6.1.8]{chang-keisler} or \cite[Theorem 13.9]{comfort-negrepontis}) that if $\U$ is
countably incomplete and $\aleph$-good, then every family of less
than $\aleph$ internal subsets of $\langle S_i\rangle_\U$ having
the finite intersection property has nonempty intersection.\\

Let $(B^\alpha)_{\alpha\in \Gamma}$ be a family of mutually
intersecting balls in $[X_i]_\U$, with $|\Gamma|<\aleph$. Let
us write $B^\alpha=B(x^\alpha,r_\alpha)$ and let $(x_i^\alpha)$ be
fixed representations of $x^\alpha$. Clearly,
$\langle B(x_i^\alpha,r_\alpha+1/m)\rangle _\U$ is a lifting of $B(x^\alpha,
r_\alpha+1/m)$ in the set-theoretic ultraproduct $\langle X_i\rangle _\U$. As
$\XiU$ is a Lindenstrauss space, the original family $(B^\alpha)$
has the finite intersection property (see the equivalence between (4.12) and (4.13) in \cite[Theorem 4.1]{zbook}).
This implies the same for the family of internal sets
$$
(\langle B(x_i^\alpha, r_\alpha+1/m)\rangle _\U)_{(\alpha,m)\in\Gamma\times\N}.
$$
Indeed, if $F$ is a finite subset of $\Gamma\times \N$, we may
assume it is of the form $E\times \{1,\dots,k\}$ for some finite
$E\subset\Gamma$. Then there exists $z\in\bigcap_{\alpha\in
E}B^\alpha$. Thus, if $(z_i)$ is a representative of $z$, the sets
$\{i\in I:\|x_i^\alpha-z_i\|\leq 1/k\}$ belong to $\U$ for every
$\alpha\in E$ and $ \langle
(z_i)\rangle_\U\in\bigcap_{(\alpha,m)\in F}\langle B(x_i^\alpha,
r_\alpha+1/m)\rangle_\U$.\\

Since $|\Gamma\times\N|< \aleph$ and $\U$ is $\aleph$-good,
there is $x\in \langle X_i\rangle_\U$ in the nonempty intersection
$$
\bigcap_{(\alpha,m)\in \Gamma\times\N}\langle B(x_i^\alpha,
r_\alpha+1/m)\rangle_\U.
$$
It is clear that if $(x_i)$ is any representation of $x$, then $$
[(x_i)]\in\bigcap_{\alpha,m}B(x^\alpha,
r_\alpha+1/m)=\bigcap_{\alpha\in \Gamma}B^\alpha,
$$
which completes the proof.
\end{proof}

A combination of \cite{accgm1} and \cite{accgm2}, see also
\cite{wrong}, shows that there exist universally $1$-separably
injective spaces not isomorphic to any $C(K)$ space. A higher
cardinal generalization is as follows.

\begin{example}\label{generalnotCK}
For every cardinal $\aleph$ there exists a space of density
$2^\aleph$ such that
\begin{itemize}
\item[(a)] It is $(1,\aleph^+)$-injective but it is not isomorphic
to a complemented subspace of any $\mathcal M$-space. \item[(b)]
After suitable renorming, it is still $(1,\aleph^+)$-injective and
its unit ball has extreme points.
\end{itemize}
\end{example}

\begin{proof}$\;$
(a) Let $\mathcal G$ be the Gurari\u \i\  space. If $\U$ is a
countably incomplete good ultrafilter on a set of cardinality
$\aleph$, then $\mathcal G_{\U}$ is an $(1,\aleph^+)$-injective
Banach space of density $2^\aleph$ by Theorem \ref{lind2}. The
assertion in (a) now follows from \cite{accgm2}.

(b) The space of Gurari\u \i\ is isomorphic to the space $A(P)$ of
continuous affine functions on the Poulsen simplex as proved by
Lusky \cite{lusky-jfa}. See also \cite{fonf,lusky-survey}. Hence
$\mathcal G_{\U}$ is isomorphic to $A(P)_\U$, in turn isometric to
the space of continuous affine functions on certain simplex $S$,
by \cite[Proposition~2.1]{heinrichL1}. Thus, the unit ball of
$A(S)=A(P)_\U$ has extreme points: $1_S$ is one. However, $A(S)$,
being isomorphic to $\mathcal G_\U$ cannot be complemented in an
 $\mathcal M$-space. As before, the density character of  $A(S)$ equals $2^\aleph$ and $A(S)$ is
$(1,\aleph^+)$-injective.
\end{proof}

The preceding examples are as bad as the generalized continuum
hypothesis allows. Indeed, if a Banach space is
$(1,\aleph^+)$-injective and has density character $\aleph$, then
it is 1-injective and then isometric to a $C(K)$-space; see for instance \cite[Corollary 1]{neville}.
The presence of an extreme point in Part (b) is reminiscent from the early studies on injectivity
(cf. \cite{nachbin, kell, a-p}).

\section{Projectiveness properties of compact spaces}\label{projectiveness}

The compact spaces arising in Theorem \ref{people}
constitute a well known class \cite{bankston,
bankston-survey,swardson} we consider now.

\begin{defin}
A compact space $K$ is said to be an $F_{\aleph}$-space
if every couple of disjoint
open subsets of $K$ which are the union of less than $\aleph$ many closed sets have disjoint
closures.
\end{defin}

The $F_{\aleph_1}$-spaces are called simply $F$-spaces. Regarding
Theorem \ref{people} let us mention that a topological space is
called $(Q_{\aleph})$-space \cite{a-p} if every couple of disjoint
open subsets of $K$ which are the union of less than $\aleph$ many
closures of open sets have disjoint closures. This property is
formally weaker than condition $(d)$ in Theorem~\ref{people} and
stronger than $(e)$ -- because every cozero set is the union of countably many closures of open sets,
$f^{-1}(\mathbb{R}\setminus\{0\}) = \bigcup \overline{f^{-1}(\mathbb{R}\setminus[-1/n,1/n])}$ --,
so it is actually equivalent to both of them in the case of compact spaces.\\

The following proposition generalizes a result of Bankston
\cite[Theorem 2.3.7(ii)]{bankston} and solves \cite[Question
2.3.8]{bankston} by showing that the extra condition of being
Boolean is not necessary. Its proof is immediate after
Theorems~\ref{people} and \ref{lind2}.

\begin{prop}\label{Fk}
Every topological ultracoproduct via a countably incomplete, $\aleph$-good ultrafilter
is an $F_{\aleph}$-space.
\end{prop}

As we mentioned before, a Banach space is 1-injective if and only
if it is isometrically isomorphic to $C(K)$ for some extremely
disconnected compact space $K$ (see \cite[Theorem 2.1]{zbook}) and such compacta
are precisely the projective elements in the category of compacta and continuous maps,
a classical result by Gleason \cite{gleason} that can be seen in \cite[Theorem 10.51]{walker}.
Which means that if $\sigma: L\to M$ is a
continuous surjection then any continuous map $\vp:K\to M$ lifts
to $L$ in the sense that there is $\tilde \vp:K\to L$ such that
$\vp=\sigma\circ\tilde\vp$. Of course this can be rephrased by
saying that $C(K)$ is injective in the category of commutative
C*-algebras. One may wonder if some natural relativization of this
result holds, meaning whether the fact that the space $C(K)$ is
injective with respect to a subcategory of Banach spaces is
reflected dually by $K$ being projective with respect to some
subcategory of compact spaces. If $\mathscr C$ is some class of
continuous surjections between compact spaces, we say that a
compact space $K$ is projective with respect to $\mathscr C$ if for
every continuous surjection $\pi:L\To M$ that belongs to $\mathscr C$
and every continuous map $f:K\To M$ there exists a continuous
function $g:K\To M$ such that $\pi g = f$. The first guess would
be that $C(K)$ being injective with respect to Banach spaces of
density less than $\aleph$ should be equivalent to the Banach
space $K$ being projective with respect to compact spaces of
weight less than $\aleph$. There is however a serious obstruction
for this approach: if $\pi$ is any surjection from the Cantor set $\Delta$ onto
the unit interval $\I$ and $K$ is any connected $F$-space, then
the only liftable maps $f:K\to\I$ are the constant ones (Proposition \ref{Fk} provides
a good number of such spaces: it is not difficult to realize that ultracoproducts
preserve connectedness since a compact space $K$ is connected if and only if the only
idempotents in $C(K)$ are 0 and 1). There are
two ways of avoiding this problem. The first way is to assume $K$
to be totally disconnected or, which is the same, zero-dimensional (Theorem \ref{metricallyprojective}).
The other way is to reduce the subcategory we are dealing with and
to consider only compact convex sets and affine maps between them
(Theorem \ref{affinemetricallyprojective}). Before going further
let us remark:

\begin{lema}
Let $\aleph$ be a cardinal number, and $K$ a compact space. The
following are equivalent:
\begin{enumerate}
\item Every open cover of every subspace of $K$ has a subcover of
cardinality less than $\aleph$. \item Every open subset of $K$ is
the union of less than $\aleph$ many closed subsets of $K$.
\end{enumerate}
\end{lema}

\begin{proof}
Suppose (1) holds and let $U$ be an open subset of $K$. Simply
consider an open cover of $U$ by open sets $V$ with
$\overline{V}\subset U$. Conversely, assume (2) and let $S\subset
K$ and $\{U_i : i\in I\}$ a cover of $S$ by open subsets of $K$.
Consider $U=\bigcup_{i\in I}U_i$. By (2), $U$ is the union of less
than $\aleph$ many compact sets, so it is enough to take a finite
subcover of each.
\end{proof}

We denote by $H\!L_{\aleph}$ the class of compact spaces
satisfying the conditions of the preceding lemma. Observe that
this class is stable under continuous images and that it contains
all compact spaces of weight less than $\aleph$. A compact space
belongs to $H\!L_{\aleph_1}$ if and only if it is hereditarily
Lindel\"{o}f, if and only if it is perfectly normal. An example of
a hereditarily Lindel\"{o}f space of uncountable weight is the
double arrow space: the lexicographical product of ordered sets
$[0,1]\times\{0,1\}$ endowed with the order topology. The
equivalence of (1), (2) and (3) in the next result is due to
Neville and Lloyd \cite{nevillealeph}. The fourth condition states
that $F_{\aleph}$ spaces are projective with respect to a larger
class of spaces than those of weight less than $\aleph$.

\begin{teor}\label{metricallyprojective}
For a compact space $K$ the following are equivalent:
\begin{enumerate}
\item $K$ is a zero-dimensional $F_{\aleph}$-space.
\item $K$ is projective with respect to
surjections $\pi:L\To M$ such that $w(L)<\aleph$.
\item $K$ is projective with respect to surjections $\pi:L\To M$
such that $w(M)<\aleph$ and $w(L)\leq\aleph$.
\item $K$ is projective with
respect to surjections $\pi:L\To M$ with $L\in H\!L_{\aleph}$.
\end{enumerate}
\end{teor}

\begin{proof} Notice that the implications $(4)\Rightarrow (2)$ and $(3)\Rightarrow (2)$ are trivial.
We prove first that (2) implies (1). In order to show that $K$ is a zero-dimensional $F_{\aleph}$-space
we shall show that for any disjoint open subsets $A$ and $B$,
which are the union of $\kappa<\aleph$ many closed subsets of $K$
there exists a clopen set $C$ such that $A\subset C$ and
$B\cap C=\varnothing$.
Suppose $A=\bigcup_{\alpha<\kappa} C_\alpha$ and
$B=\bigcup_{\alpha<\kappa}D_\alpha$ where each $C_\alpha$ and
each $D_\alpha$ are closed sets.
For every $\alpha<\kappa$ let $f_\alpha:K\To [-1,1]$ be a continuous
function such that
\begin{itemize}
\item $f_\alpha|_{C_\alpha}= -1$,
\item $f_\alpha|_A\leq 0$,
\item $f_\alpha|_{K\setminus (A\cup B)}= 0$,
\item $f_\alpha|_B\geq 0$, and
\item $f_\alpha|_{D_\alpha}= 1$.
\end{itemize}

Consider the map $f:K\To [-1,1]^\kappa$ given by $f(x) =
(f_\alpha(x))_{\alpha<\kappa}$. Let also
$L=[0,1]^\kappa\times\{-1,1\}$, $\pi:L\To [-1,1]^\kappa$ be given
by $\pi(x,t)=(t\cdot x_\alpha)_{\alpha<\kappa}$ and $M=\pi(L)$.
Notice that the image of $f$ is contained in $M$, hence we are in
a position to apply the projectiveness property so that there exists
$g:K\To L$ with $\pi g=f$. But then $g(A)\subset
[0,1]^\kappa\times\{-1\}$ and $g(B)\subset
[0,1]^\kappa\times\{1\}$, hence there are disjoint clopen sets
which separate $A$ and $B$. Conversely, we prove now that $(1)$ implies $(3)$ and $(4)$.
So assume now that $K$ is a zero-dimensional $F_{\aleph}$ space.
We assume that we are given an onto map $\pi:L\To M$ like either in (3) or (4), and $f:K\To M$,
and we will find $g:K\To L$ with $\pi g = f$.
\smallskip

{\sc Case 1.} We suppose that $M\in H\!L_{\aleph}$, $L\subset
M\times\{0,1\}$ and $\pi:L\To M$ is the first-coordinate
projection. Consider $$A =
K\setminus f^{-1}[\pi(L\cap M\times\{1\})],$$
$$B=K\setminus
f^{-1}[\pi(L\cap M\times\{0\})].$$ These are two disjoint open subsets of
$K$ which are moreover the union of less than $\aleph$ many closed
sets, because $M\in H\!L_{\aleph}$. Therefore, since $K$ is a
totally disconnected $F_{\aleph}$ space, there exists a clopen set
$C\subset K$ such that $A\subset C$ and $B\subset C=\varnothing$.
The desired function $g:K\To L$ can be defined now as $g(x)=(x,0)$
if $x\in C$ and $g(x)=(x,1)$ if $x\not\in C$.

\smallskip

{\sc Case 2.} We suppose that $L\in H\!L_{\aleph}, L\subset
M\times[0,1]$ and $\pi:L\To M$ is the first-coordinate projection.
Let $q:2^\omega\To [0,1]$ be a continuous surjection from the
Cantor set onto the unit interval. Let $L' = \{(x,t)\in M\times
2^\omega : (x,q(t))\in L \}$ and $\pi':L'\To M$ the first
coordinate projection. We shall find a continuous map $g':K\To L'$
such that $\pi'g'=f$. From $g'$ we easily obtain the desired
function $g$ by composing with $q$ in the second coordinate. For
every $n<m\leq\omega$ let $p_n^m:M\times 2^m\To M\times 2^n$ be
the natural projection which forgets about coordinates $i\geq n$
in $2^m$. Let $L_n = p_n^\omega(L')$. Each $L_n\subset L\times
2^n$ is a member of $H\!L_{\aleph}$. Hence, by repeated
application of the Case 1 proved above, we can construct
inductively continuous maps $g_n:K\To L_n$ such that $g_0=f$ and
$\pi_{n}^{n+1} g_{n+1} = g_n$. These functions must be of the form
$g_n(x) = (f(x),\gamma_0(x),\ldots,\gamma_{n-1}(x))$ for some
continuous functions $\gamma_i:K\To 2$, $i<\omega$. The function
$g':K\To L'\subset M\times 2^\omega$ is defined as $g'(x) =
(f(x),\gamma_0(x),\gamma_1(x),\ldots)$.
\smallskip

{\sc General case.} We view $L$ as a closed subset of a cube
$L\subset 2^\Gamma$, where $\Gamma$ is some cardinal. If we are
dealing with condition (3), then $\Gamma=\aleph$. Let
$G=\{(x,\pi(x)) : x\in L\}\subset [0,1]^\Gamma\times M$ be the
graph of $\pi$, and let $\pi_1:G\To L$ and $\pi_2:G\To M$ be the
two coordinate functions. We shall find a continuous function
$h:K\To G$ such that $\pi_2 h=f$. From this we immediately get the
desired lifting as $g=\pi_1 h$.

\smallskip

For every $\alpha<\beta\leq \Gamma$ let
$p_\alpha^\beta:2^\beta\times M\To 2^\alpha\times M$ be the
natural projection and let $G_\alpha = p_\alpha^\Gamma(G)$. If we
assume condition (3) then all spaces $G_\alpha$ have weight less
than $\aleph$, while if we assume (4), then all these spaces
belong to $H\!L_{\aleph}$ because $G$ is
homeomorphic to $L$ and this class is stable under taking
continuous images. We construct by transfinite induction
continuous functions $h_\alpha:K\To G_\alpha$ such that $\pi_2
h_\alpha=f$ and such that they are coherent: $p_\alpha^\beta
h_\beta = h_\alpha$ for $\alpha<\beta$. In the one immediate
successor step of the induction, in order to obtain $h_{\alpha+1}$
from $h_\alpha$ we are in a position to apply Case 2 above. In the
limit step, the function $h_\beta$ is uniquely determined by the
functions $h_\alpha$ with $\alpha<\beta$, similarly as we did in
Case 2.
\end{proof}

\begin{cor} The following spaces are ``projective'' with respect to all continuous
surjections between metrizable compacta
\begin{itemize}
\item
$\mathbb{N}^\ast$, the growth of the integers in its Stone-\v Cech compactification.
\item
Ultracoproducts of families of totally disconnected compacta built over countably incomplete ultrafilters.
\end{itemize}
\end{cor}

\begin{cor}\label{HLprojective}
Totally disconnected $F$-spaces are projective with respect to
hereditarily Lindel\"{o}f compact spaces.
\end{cor}

Some particular cases of Corollary \ref{HLprojective} are proven
by Przymusi\'{n}ski \cite{Przymusinki} to the effect of showing
that every hereditarily Lindel\"{o}f compact space is a continuous
image of $\mathbb{N}^\ast$. Yet his arguments require some extra
hypotheses which are unnecessary at the end. In the following
Corollary, we denote by $RO(X)$ the set of all regular open
subsets of $X$, that is, those open sets which are interiors of
closed sets.

\begin{cor}
Let $K$ be a totally disconnected $F_{\aleph}$-space. Then $K$ is
projective with respect to surjections $\pi:L\To M$ in which
$w(M)<\aleph$ and $|RO(M)|\leq \aleph$.
\end{cor}

\begin{proof}
Let $f:K\To L$ as usual, and let $p:G\To M$ be the Gleason cover of $M$. We refer to
\cite{walker} for an explanation about Gleason covers. We just
recall the facts that we need about it: the space $G$ is an
extremely disconnected space (that is, projective with respect to
the full category of compact spaces), $w(G)=|RO(M)|$ and, and
$p:G\To M$ is an onto continuous map. Since $w(G)\leq \aleph$ and
$w(M)<\aleph$, by Theorem \ref{metricallyprojective} there exists
$h:K\To G$ such that $ph=f$. Since $G$ is projective, there exists
$u:G\To L$ such that $\pi u =p$. Take $g= uh$.
\end{proof}

\begin{cor}[Neville and Lloyd]
If $\kappa$ is a cardinal for which $\kappa^+=2^\kappa$, and $K$
is a totally disconnected compact $F_{\kappa^+}$-space, then $K$
is projective with respect to all surjections $\pi:L\To M$ such
that $w(M)\leq\kappa$.
\end{cor}

\begin{proof} Apply the preceding Corollary for $\aleph=\kappa^+$,
and notice that one always has
$
|RO(M)|\leq 2^{w(M)}
$
because every open set is the union of a family of open sets
from a basis.
\end{proof}

Neville and Lloyd \cite{nevillealeph} asked whether the assumption
that $\kappa^+=2^\kappa$ can be removed. We point out that the
compact space constructed by Dow and Hart \cite{dowhart} we used
in \cite[Theorem 7]{accgm1} provides a negative answer to their
question.

\begin{teor}
It is consistent that there exists a zero-dimensional
compact $F$-space $K$ which is not projective with respect
to surjections $\pi:L\To M$ with $w(M)=\aleph_0$.
\end{teor}

\begin{proof}
Under the assumption that $\mathfrak c=\aleph_2$ and that
$\wp(\N)/\fin$ contains a chain of order type $\omega_2$, Dow and
Hart \cite[Theorem 5.10]{dowhart} construct a zero-dimensional
compact
 $F$-space $K$ which does not map onto $\beta\N$.
Let $M=\alpha\N$ be the one-point compactification of
the natural numbers, $L=\beta\N$ and $\pi:\beta\N\To M$ defined as
$\pi(n)=n$ for $n\in\N$, and $\pi(x)=\infty$ if
$x\in\beta\N\setminus\N$. Let $f:K\To M$ be a continuous
surjection. We claim that any continuous map $g:K\To L$ with $\pi
g=f$ must be onto, hence there is no such $g$. The reason is that
for every $n\in\N$, if $x_n$ is such that $f(x_n)=n$, $\pi
g(x_n)=n$, hence $g(x_n)=n$. Therefore $\N\subset g(K)$, and since $\N$
is dense in $L$, we conclude that $g$ is onto, as desired.
\end{proof}

In the next Theorem, by a compact convex set we mean a compact
convex set lying inside some locally convex space $E$.
Actually, every such set $L$ is affinely homeomorphic to a closed
convex subset of a cube $[0,1]^\Gamma$, where the size of $\Gamma$
can be as small as the weight of $L$. This is a consequence of the fact that
continuous linear functionals on $E$ separate points \cite[Corollary 3.33]{checos}:
One takes takes $\Gamma$ as the set of these functionals and then the correspondence
$x\mapsto (f(x))_{f\in\Gamma}$ shows that $L$ is affinely homeomorphic to a compact
convex subset of $\mathbb{R}^\Gamma$, indeed by compactness to a subset of a
product of intervals $\prod_\Gamma[a_\gamma,b_\gamma]$, which is in turn affinely
homeomorphic to $[0,1]^\Gamma$.

\begin{teor}\label{affinemetricallyprojective} Suppose $\aleph\geq \aleph_1$.
For a compact space $K$ the following are equivalent
\begin{enumerate}
\item $K$ is an $F_{\aleph}$-space \item For every continuous
affine surjection $\pi:L\To M$ between compact convex sets with
$w(L)<\aleph$, and every continuous function $f:K\To M$, there
exists a continuous function $g:K\To L$ such that $\pi g = f$.
\item As above with $w(M)<\aleph$ and $w(L)\leq\aleph$. \item As
above with $L\in H\!L_{\aleph}$.
\end{enumerate}
\end{teor}

\begin{proof}
It is clear that $(3) \Rightarrow (2)$ and that $(4) \Rightarrow (2)$.
We shall prove that (2) implies (1), and that (1) implies (3) and (4).
Suppose first that (2) holds, and we shall show that the second condition
of Theorem \ref{people} holds for any cardinal $\Gamma<\aleph$.
Let $f_\alpha, g_\alpha:K\To [0,1]$, with $\alpha<\Gamma$, be two families
of continuous functions such that $f_\alpha\leq g_\beta$ for every
$\alpha,\beta<\Gamma$.
Consider
\begin{eqnarray*}
M &=&
\left\{\left((t_\alpha)_{\alpha<\Gamma},(s_\alpha)_{\alpha<\Gamma}\right)
\in [0,1]^\Gamma\times [0,1]^\Gamma :
\sup_{\alpha<\Gamma} t_\alpha \leq \inf_{\alpha<\Gamma} s_\alpha\right\},\\
L &=& \left\{\left((t_\alpha)_{\alpha<\Gamma},r,(s_\alpha)_{\alpha<\Gamma}\right)\in [
0,1]^\Gamma\times [0,1]\times [0,1]^\Gamma :
\sup_{\alpha<\Gamma} t_\alpha \leq r\leq \inf_{\alpha<\Gamma} s_\alpha\right\}.
\end{eqnarray*}
Let $\pi:L\To M$ be the natural surjection which forgets the
intermediate coordinate $r$, and let also $f:K\To M$ be given by
$$
f(x) = (f_\alpha(x)_{\alpha<\Gamma},g_\alpha(x)_{\alpha<\Gamma}).
$$
We are in a position to apply the statement of part (2), so that
there is a function $g:K\To L$ such that $\pi(g(x)) = f(x)$. If we
look at the composition of $g$ with the projection on the central
coordinate $r$ of $L$, we obtain a continuous function $h:K\To
[0,1]$ such that $f_\alpha\leq h \leq g_\alpha$ for every
$\alpha<\Gamma$. This proves that $K$ is an $F_{\aleph}$-space.\\

Now we proceed to the proof that (1) implies (3) and (4), so we suppose that $K$
is an $F_{\aleph}$-space, $\pi:L\To M$ is a continuous affine surjection
and $f:K\To M$ is a continuous surjection. We want to show that, under the hypotheses
of either (3) or (4), we get a continuous function $g:K\To L$ such that $\pi g = f$.
We consider $M$ to be a closed convex subset of a cube, $M\subset [0,1]^\Gamma$
and we call $\pi_\alpha:M\To [0,1]$ to the projection on the
$\alpha$-th coodinate. The first step is to find the desired
function $g$ under the following assumption (which can be
considered the analogue of considering a Banach superspace of
codimension 1):\\

{\sc Step 1.} We assume $M\in H\!L _{\aleph}$ and there exists a
continuous affine function $\phi:L\To [0,1]$ such that the map
$(\pi,\phi) : L\To M\times [0,1]$ given by $(\pi,\phi)(x) =
(\pi(x),\phi(x))$ is one-to-one.\\

In this case, we shall view $L$ as a closed convex subset of
$M\times [0,1]$, so that $\pi$ and $\phi$ are just the projections
on the first and second coordinate. To find the desired function
$g:K\To L$ is equivalent to find a continuous function
$\gamma:K\To [0,1]$ such that $(f(x),\gamma(x))\in L$ for every
$x\in K$. Let $\{q_n : n<\omega\}$ be a countable dense subset of
$[0,1]$. We shall define by induction continuous functions
$\gamma_n^-,\gamma_n^+:K\To [0,1]$ such that $\gamma_n^-
\leq\gamma_m^+$ for every $n,m$, and then $\gamma$ will be chosen
such that $\gamma_n^- \leq\gamma\leq\gamma_m^+$ for every $n,m$.
For each $n$, define
\begin{eqnarray*}
U_n^- &=& \{y\in M :  (y,t)\not\in L \text{ for every } t\in [q_n,1]\}\setminus \pi(L\cap (M\times [q_n,1])), \\
U_n^+ &=& \{y\in M : (y,t)\not\in L  \text{ for every } t\in [0,q_n]\}\setminus \pi(L\cap (M\times [0,q_n])),
\end{eqnarray*}
which are two disjoint open subsets of $M$.
Since $M\in H\!L_{\aleph}$, $f^{-1}(U^-_n)$ and $f^{-1}(U^+_n)$
are disjoint open subsets of $K$ which are moreover unions of less than $\aleph$ many
closed sets.
Since $K$ is an $F_{\aleph}$-space, there exist continuous functions
$\delta_n^-$ and $\delta_n^+$ over $K$ such that
\begin{eqnarray*}
0\leq \delta_n^-\leq q_n,  \ \  \delta_n^-|_{f^{-1}(U^-_n)}\equiv 0,
\ \ \delta_n^-|_{f^{-1}(U^+_n)}\equiv q_n,\\
q_n\leq \delta_n^+\leq 1, \ \  \delta_n^+|_{f^{-1}(U^-_n)}\equiv q_n,
\ \  \delta_n^+|_{f^{-1}(U^+_n)}\equiv 1.
\end{eqnarray*}

A priori, it may be false that $\delta_n^-\leq \delta_m^+$ for every
$n,m$, so in order to make sure of this we define inductively:
\begin{eqnarray*}
\gamma_n^- = \min\{\delta_n^-, \gamma_m^+ : m<n\},
\ \ \gamma_n^+ = \max\{\delta_n^+, \gamma_m^- : m<n\}.
\end{eqnarray*}
It is easy to see (using the fact that if $q_i<q_j$, then
$f^{-1}(U^-_i)\subset f^{-1}(U^-_j)$ and
$f^{-1}(U^+_i)\supset f^{-1}(U^+_j)$) that these new functions
still keep the key properties that
\begin{eqnarray*}
0\leq \gamma_n^-\leq q_n,  \ \  \gamma_n^-|_{f^{-1}(U^-_n)}\equiv 0,
\ \ \gamma_n^-|_{f^{-1}(U^+_n)}\equiv q_n,\\
q_n\leq \gamma_n^+\leq 1, \ \  \gamma_n^+|_{f^{-1}(U^-_n)}\equiv q_n,
\ \  \gamma_n^+|_{f^{-1}(U^+_n)}\equiv 1.
\end{eqnarray*}

Since $K$ is in particular an $F$-space,
there exists a continuous function $\gamma:K\To [0,1]$  such that
$\gamma^-_n\leq \gamma \leq \gamma^+_n$ for all $n$.
We have to show that $(f(x),\gamma(x))\in L$ for every $x\in K$.
Given $x\in K$, let
$$
I = \{t\in [0,1] : (f(x),t)\in L\} = \phi(\pi^{-1}[f(x)]).
$$

Since $\phi$ and $\pi$ are affine, $I=[a,b]$ is a closed interval.
In order to check that $\gamma(x)\in I$, we show that $q_n\leq
\gamma(x)\leq q_m$ whenever $q_n<a$ and $q_m>b$. For example, if
$q_n<a$, then this means that $f(x)\in U_n^+$, $x\in
f^{-1}(U_n^+)$, so $q_n = \gamma_n^-(x)\leq \gamma(x)$.
Analogously, if $q_m>b$, then $x\in f^{-1}(U_m^-)$, and
$\gamma(x)\leq \gamma^+_m(x)=q_m$. This finishes the proof under
the assumption made in Step 1.\\

{\sc General case.} We view now $L$ as compact convex set of the
Hilbert cube $[0,1]^\Gamma$ (with $\Gamma = \aleph$ when we are
under the assumptions of case (3)) and we call $\chi_\alpha:L\To
[0,1]$ the coordinate functions, $\alpha<\Gamma$. For every
$\alpha$, we consider the map $h_\alpha:L\To M\times [0,1]^\alpha$
given by $h_\alpha(z) = (\pi(z),\chi_\beta(z)_{\beta<\alpha})$,
and we call $L_\alpha = h_\alpha(L)\subset M\times [0,1]^\alpha$
the image of this continuous function. For $\alpha<\beta$, we also
call $p_\alpha^\beta:L_\beta\To L_\alpha$
 the continuous surjection which forgets about coordinates $t_i$
with $i\geq \alpha$.
We construct by transfinite induction a sequence of coherent liftings
$g_\alpha:K\To L_\alpha$, $\alpha<\Gamma$, that is, functions
satisfying $g_0=f$ and $p_\alpha^\beta g_\beta = g_\alpha$
whenever $\alpha<\beta$.
Notice that this is actually equivalent to finding continuous functions
$\gamma_\alpha:K\To [0,1]$ such that
$g_\alpha(x)=(f(x),\gamma_\beta(x)_{\beta<\alpha})\in L_\alpha$ for
every $x\in K$ and $\alpha\leq\Gamma$.
In the inductive process $g_{\alpha+1}$ is obtained from $g_\alpha$
by applying Step 1, while in the limit ordinals one has to take
$g_\beta(x) = (f(x),\gamma_\alpha(x)_{\alpha<\beta})$.
Notice that Step 1 can be applied because $L_\alpha\in H\!L_\aleph$:
if we are in case (3), we took $\Gamma=\aleph$, so $w(L_\alpha)<\aleph$, while
in case (4) $L_\alpha$ is a continuous image of $L$ and $L\in H\!L_\aleph$.
Let $g_\Gamma:K\To M\times L$ be the final output of this inductive
construction. We have that $p_0^\Gamma g_\Gamma=f$.
Let $g:K\To L$ be obtained by projecting $g_\Gamma$ on the second
coordinate, so that we can write $g^\Gamma(x) = (f(x),g(x))$.
The fact that $g^\Gamma(x)\in L_\Gamma$ implies that $\pi(g(x)) = f(x)$, so $g$
is the map that we were looking for.
\end{proof}

The fact that $C(K)$ is ($1, \aleph$)-injective when  $K$ is an
$F_{\aleph}$-space is a consequence of Theorem~\ref{affinemetricallyprojective}.
Suppose we have $Y\subset X$ Banach spaces with $\dens X<\aleph$ and
$t:Y\To C(K)$ an operator of norm 1. We can apply part (2) of
Theorem~\ref{affinemetricallyprojective} to $\pi:B_{X^{\ast}}\To B_{Y^{\ast}}$ and
the mapping $f:K\To B_{Y^{\ast}}$ given by $f(x)=t^\ast(\delta_x)$.
We obtain a weak*-continuous function $g:K\To B_{X^\ast}$ such that $\pi g=f$.
Then, the formula ${T}(x)(k) = \|x\| g(k)(x/\|x\|)$, $x\in X$, $k\in K$,
defines an extension of $t$ of norm 1.

\section{Open Problems }

\noindent (1) Find  homological characterizations of $\aleph$-injectivity
and universal $\aleph$-injectivity in $\sf{ZFC}$. Proposition
\ref{doselevado} characterizes $(2^\aleph)^+$-injectivity; in particular,
it characterizes $\aleph_2$-injectivity under {\sf
CH} (\emph{a Banach space $E$ is $\aleph_2$-injective if and only
if if it is complemented in every superspace $W$ such that $W/E$
is a quotient of $\ell_\infty$}) and $\aleph^+$-injectivity under
{\rm {\sf{GCH}} (\emph{a Banach space $E$ is $\aleph^+$-injective
if and only if it is complemented in every superspace $W$ such
that $W/E$ is a quotient of $\ell_\infty(\aleph)$}}).
\smallskip

\noindent (2) Find a characterization of universal $(1,\aleph)$-injectivity
by means of intersection of families of balls.
\smallskip

\noindent (3) Is universal $\aleph$-injectivity a $3$-space property?
\smallskip

\noindent (4) Is it consistent that $C(\N^*)$ is $\aleph_2$-injective?
Recall that we have already shown that $C(\N^*)$ is not
$(1,\aleph_2)$-injective nor $\mathfrak c^+$-injective.
\smallskip

\noindent (5) Are Lindenstrauss ultraproducts via countably incomplete $\aleph$-good
ultrafilters universally $\aleph$-injective spaces in {\sf{ZFC}}?
(They are universally $(1,\aleph)$-injective under {\sf GCH}.)
\smallskip

\noindent (6) Prove or disprove that every ultraproduct built over a countably incomplete,
$\aleph$-good ultrafilter is $\aleph$-injective as long as it is
a $\mathscr L_\infty$-space.

\end{document}